\documentclass{amsart}
\usepackage{bbm}
\usepackage{amsfonts}
\usepackage{mathrsfs}
\usepackage{amsmath,bm}

\newtheorem{lem}{Lemma}[section]

\newtheorem{thm}[lem]{Theorem}
\newtheorem{cor}[lem]{Corollary}
\theoremstyle{remark}
\newtheorem{rem}[lem]{Remark}

\DeclareMathOperator\card{Card}
\numberwithin{equation}{section}

\begin{document}
\title[Hausdorff dimensions of level sets related to moving digit means]{Hausdorff dimensions of level sets related to moving digit means}
\author{Haibo Chen}
 \address{School of Statistics and Mathematics, Zhongnan University of Economics and Law, Wu\-han, Hubei, 430073, China}
  \email{hiboo\_chen@sohu.com}

\begin{abstract}
In this paper, we will introduce and study the lower moving digit mean $\b{\it M}(x)$ and the upper moving digit mean $\bar{M}(x)$ of $x\in[0,1]$ in $p$-adic expansion, where $p\geq2$ is an integer. Moreover, the Hausdorff dimension of level set  
\[B(\alpha,\beta)=\left\{x\in [0,1]\colon \b{\it M}(x)=\alpha,\bar{M}(x)=\beta\right\}\]
is determined for each pair of numbers $\alpha$ and $\beta$ satisfying with $0\leq\alpha\leq\beta\leq p-1$.
\end{abstract}

\maketitle

\section{Introduction}
To determine the Hausdorff dimensions of sets of numbers in which the distributions of the digits are of specific characters for some representation is a fundamental and important problem in number theory and multifractal analysis. It has a long history and there are a great many classic work on this topic, suh as Besicovitch~\cite{Be}, Eggleston~\cite{E}, Billingsley~\cite{Bi}, Barreira et al.~\cite{BSS} and Fan et al.~\cite{FF,FFW}, in which the sets are associated with frequencies of digits of numbers in different expansions, long-term time averages and ergodic limits, etc.
Actually, the expressions in these sets are of a common feature, i.e., they are usually described by the arithmetic mean. Different from this, in this paper we would like to study the moving digit means of numbers and investigate the corresponding level sets related to it in the unit interval. 

The moving digit mean is derived from the concept of tangential dimension at a point for measure studied by Guido and Isola~\cite{GI1,GI2}. To be precise, if the measure is a Bernoulli measure, then the tangential dimension is of a linear relation with the moving digit mean (see~\cite{CWX}). More significantly, when we explore the multifractal behavior at a point of a measure, the tangential dimension is more sensitive than the local dimensions of measures, which enables the tangential dimension to provide more information than the local dimension. Thus, compared with the arithmetic mean, the moving digit mean used in the description of level sets can also provide more information on the distributions of digits of numbers. In the following, we introduce the corresponding concepts and notations.

Let $p\geq 2$ be an integer and $A=\{0,1,\ldots p-1\}$ be the alphabet with $p$ elements. It is known that each number $x\in I= [0,1]$ can be expanded into an infinite non-terminating expression  
\[\sum_{n=1}^\infty\frac{x_n}{p^n}=0.x_1x_2x_3\ldots,\quad\text{where}\ x_n\in A,\ n\geq1,\] 
which is called the $p$-adic expansion of $x$. Denoted by $S_n(x)=\sum_{i=1}^nx_i, n\geq1$, the $n$-th partial sum of $x$. Let $T\colon I\to I$ be the shift operator defined by
\[Tx=0.x_2x_3x_4\ldots,\quad\text{for any}\ x=0.x_1x_2x_3\ldots\in I.\]  
Let $x\in I$, we call
\begin{align}
  \b{\it M}(x)=\lim_{n\to\infty}\varliminf_{m\to\infty}\frac{S_n(T^mx)}{n}\quad\text{and}\quad\bar{M}(x)=\lim_{n\to\infty}\varlimsup_{m\to\infty}\frac{S_n(T^mx)}{n}
\end{align}	
\emph{the lower and upper moving digit means} of $x$ respectively. 

\begin{rem}
	Let $n\geq1$. Denote by
	\[\b{\it M}_n(x)=\varliminf_{m\to\infty}\frac{S_n(T^mx)}{n}\quad\text{and}\quad\bar{M}_n(x)=\varlimsup_{m\to\infty}\frac{S_n(T^mx)}{n}\]
	the $n$-th lower and upper moving digit means respectively. It is easy to check that, for any $x\in I$, the two sequences $\{-\b{\it M}_n(x)\}_{n\geq1}$ and $\{\bar{M}_n(x)\}_{n\geq1}$ are both subadditive since they satisfy  the inequalities:
	\[-\b{\it M}_{m+n}(x)\leq-\b{\it M}_m(x)-\b{\it M}_n(x)\quad\text{and}\quad\bar{M}_{m+n}(x)\leq\bar{M}_m(x)+\bar{M}_n(x),\quad m,n\geq1.\] 
	Thus, the limits of $\b{\it M}_n(x)$ and $\bar{M}_n(x)$ always exist as $n\to\infty$. So, the definitions of $\b{\it M}(x)$ and $\bar{M}(x)$ are both reasonable.
\end{rem}

To investigate the influence of the lower and upper moving digit means on the distribution of digits of numbers, define the level set
\begin{align}
	B(\alpha,\beta)=\{x\in I\colon \b{\it M}(x)=\alpha,\bar{M}(x)=\beta\},\quad\text{where}\ 0\leq\alpha\leq\beta\leq p-1,
\end{align}
which may be called \emph{Banach set} with lower level $\alpha$ and upper level $\beta$. In this paper, we would like to determine the Hausdorff dimension of set $B(\alpha,\beta)$ for any $p\geq2$, which is a non-trivial and meaningful generalization of the work in \cite{CWX}. First, we give the definition of $p$-adic entropy function below.

Let $0\leq\alpha\leq p-1$ and $\delta>0$. Denote 
\[H(\alpha,n,\delta)=\left\{x_1x_2\cdots x_n\in A^n\colon n(\alpha-\delta)<\sum_{i=1}^nx_i<n(\alpha+\delta)\right\}\]
and $h(\alpha,n,\delta)=\card H(\alpha,n,\delta)$. Here and in the sequel, the symbol $\card$ is the cardinality of some finite set. Furthermore, define the $p$-adic entropy function as
\begin{align}\label{definiton h alpha}
	h(\alpha)=\lim_{\delta\to0}\varlimsup_{n\to\infty}\frac{\log h(\alpha,n,\delta)}{(\log p)n},\quad 0\leq\alpha\leq p-1.
\end{align}
It is evident that $0\leq h(\alpha)\leq1$ for any $0\leq\alpha\leq p-1$. 

Denote by $\dim_H$ the Hausdorff dimension of some set. Then we have
\begin{thm}\label{theorem main theorem}
	For any $0\leq\alpha\leq\beta\leq p-1$, we have
\begin{align}\label{formula main}
	\dim_HB(\alpha,\beta)=\sup_{\alpha\leq t\leq\beta}h(t).
\end{align}
\end{thm}

In particular, let $0\leq\alpha\leq p-1$ and write
\begin{align}
	B(\alpha):=B(\alpha,\alpha)=\{x\in I\colon \b{\it M}(x)=\bar{M}(x)=\alpha\}.
\end{align}
Then we may obtain immediately that 
\begin{cor}\label{corollary main corollary}
	For any $0\leq\alpha\leq p-1$, we have $\dim_HB(\alpha)=h(\alpha)$. 
\end{cor}

Note that we will discuss another set $B^\ast(\alpha)$ which is related to the moving digit mean $M(x)$ in the end of this paper. As a result, $B(\alpha)$ is   different from and large than $B^\ast(\alpha)$ in the sense of Hausdorff dimension.

Moreover, in the case $p=2$, we can give an explicit expression to the binary entropy function $h_2$ by the following calculation:
\begin{align}
	  h_2(\alpha)=\lim_{\delta\to0}\varlimsup_{n\to\infty}\frac{\log \sum_{i=[n(\alpha-\delta)]+1}^{[n(\alpha+\delta)]}\binom{n}{i}}{(\log 2)n}=\frac{-\alpha\log\alpha-(1-\alpha)\log(1-\alpha)}{\log2},
\end{align}
where $\binom ni$ is the binomial coefficient. Thus, the following conclusion can be easily obtained by Theorem \ref{theorem main theorem} and Corollary \ref{corollary main corollary}. 

\begin{cor}
	Let $p=2$. If $0\leq\alpha\leq\beta\leq 1$, then $\dim_HB(\alpha,\beta)=\sup_{\alpha\leq t\leq\beta}h_2(t)$. If $0\leq\alpha\leq1$, then $\dim_HB(\alpha)=h_2(\alpha)$.
\end{cor}

In the present paper, the readers are assumed to know well the definitions and basic properties of Hausdorff dimension and Hausdorff measure. For these and more related theory, one can refer to Falconer's book~\cite{F97}. 

The structure of this paper is as follows. In the next section, the concepts of $p$-adic lower and upper entropy functions are introduced, their relations to the $p$-adic entropy function are also shown. In Section 3, the concepts Besicovitch sets in $p$-adic expansion are introduced. Moreover, their Hausdorff dimensions are also determined, which generalizes an early work of Besicovitch. A In Section 4, we will introduce some special Moran sets and determine their Hausdorff dimension for ready use. The last section is devoted to the proof of Theorem \ref{theorem main theorem}, some further discussion are also presented there. 

\section{$p$-adic entropy function}
In this section, we will introduce the definitions of $p$-adic lower entropy function $\b{\it h}(\alpha)$ and $p$-adic upper entropy function $\bar{h}(\alpha)$, and then present some properties about the two functions and the $p$-adic entropy function $h(\alpha)$.

Let $0\leq\alpha\leq p-1$, $n\geq1$ and $\delta>0$. Denote  
\[\bar{H}(\alpha,n,\delta)=\left\{x_1x_2\cdots x_n\in A^n\colon\sum_{i=1}^nx_i<n(\alpha+\delta)\right\}\]
and $\bar{h}(\alpha,n,\delta)=\card\bar{H}(\alpha,n,\delta)$. Define the $p$-adic upper entropy function
\begin{align}\label{definition bar h}
\bar{h}(\alpha)=\lim_{\delta\to0}\varlimsup_{n\to\infty}\frac{\log\bar{h}(\alpha,n,\delta)}{(\log p)n}.
\end{align}
Oppositely, denote  
\[\b{\it H}(\alpha,n,\delta)=\left\{x_1x_2\cdots x_n\in A^n\colon\sum_{i=1}^nx_i>n(\alpha-\delta)\right\}\]
and $\b{\it h}(\alpha,n,\delta)=\card\b{\it H}(\alpha,n,\delta)$. Define the $p$-adic lower entropy function
\begin{align}\label{definition b h}
  \b{\it h}(\alpha)=\lim_{\delta\to0}\varliminf_{n\to\infty}\frac{\log\b{\it h}(\alpha,n,\delta)}{(\log p)n}.
\end{align}
Note that we have $0\leq\b{\it h}(\alpha),\bar{h}(\alpha)\leq1$ and the limits in \eqref{definition bar h} and \eqref{definition b h} both exist since $\bar{h}(\alpha,n,\delta)$ and $\b{\it h}(\alpha,n,\delta)$ are increasing for $\delta>0$. Moreover, for the two functions $\b{\it h}(\alpha)$ and $\bar{h}(\alpha)$, we have the following relation between them.
\begin{thm}\label{theorem barh bh}
	$\bar{h}(p-1-\alpha)=\b{\it h}(\alpha)$.
\end{thm}
\begin{proof}
For any $n$-word $x_1\cdots x_n\in A^n$, it is easy to check that the word $(p-1-x_1)\cdots(p-1-x_n)\in\b{\it H}(\alpha)$ if and only if the word $x_1\cdots x_n\in\bar{H}(p-1-\alpha)$. So, there is a one-to-one corresponding between the two sets $\b{\it H}(\alpha)$ and $\bar{H}(p-1-\alpha)$.	
\end{proof}

Recall the definition of $p$-adic function $h(\alpha)$ in \eqref{definiton h alpha}. Actually, we may have
\begin{align}\label{definiton h alpha ==}
h(\alpha)=\lim_{\delta\to0}\varliminf_{n\to\infty}\frac{\log h(\alpha,n,\delta)}{(\log p)n}=\lim_{\delta\to0}\varlimsup_{n\to\infty}\frac{\log h(\alpha,n,\delta)}{(\log p)n} 
\end{align}
according to Proposition 4.2 in~\cite{TWWX}. Moreover, we can also discover the following properties of $p$-adic function.
 
\begin{thm}\label{theorem enumerate 1}
	For the function $h(\alpha)$ defined on $[0,p-1]$, we have
\noindent	
    \begin{enumerate}
    	\item\label{halpha 1} $h(0)=h(p-1)=0$;
		\item\label{halpha 2} $h(\alpha)$ is concave and continuous on $[0,p-1]$; 
		\item\label{halpha 3} $h(\alpha)$ is symmetric with respect to the line $\alpha=(p-1)/2$. That is, we have $h(\alpha)=h(p-1-\alpha)$ for any $0\leq\alpha\leq p-1$. It follows that $h(\alpha)$ is increasing on $[0,(p-1)/2]$ and decreasing on $[(p-1)/2,p-1]$;
		\item\label{halpha 4} If $0\leq\alpha\leq(p-1)/2$, then $\bar{h}(\alpha)=h(\alpha)$; if $(p-1)/2\leq\alpha\leq p-1$, then $\b{\it h}(\alpha)=h(\alpha)$.
	\end{enumerate}			
\end{thm}
\begin{proof}
(1) The conclusion $h(0)=0$	is followed by the the estimation
\begin{align}\label{inequality h zero}
   h(0)\leq\lim_{\delta\to0}\varliminf_{n\to\infty}\frac{\log\frac{n^{[n\delta]}}{[n\delta]!}}{(\log p)n}\leq\lim_{\delta\to0}\varliminf_{n\to\infty}\frac{\log\frac{n^{n\delta}}{(n\delta-1)!}}{(\log p)n}		=\lim_{\delta\to0}\frac{\delta\log\frac{e}{\delta}}{\log p}=0.
\end{align}
Here, the third equality is followed by the well-known Stirling's approximation and $y!=y(y-1)\cdots(y-[y])$ if $y>0$. The other conclusion $h(p-1)=0$ can be deduced similarly.
	
(2) Let $m\geq1$ and take $m$ $n$-words $X_1,X_2,\ldots,X_m\subset H(\alpha,n,\delta)$. It is obvious that the concatenation of these words satisfies $X_1X_2\cdots X_m\subset H(\alpha,nm,\delta)$. Thus, \[\big(h(\alpha,n,\delta)\big)^m\leq h(\alpha,nm,\delta).\] 
Let $\alpha,\beta\in (0,p-1)$ and $s$, $t$ be two positive integers. Then
\[\big(h(\alpha,n,\delta)\big)^s\big(h(\beta,n,\delta)\big)^t\leq h(\alpha,ns,\delta)h(\beta,nt,\delta)\leq h\left(\frac{s\alpha+t\beta}{p+q},n(s+t),\delta\right).\]
Hence, 
\[\frac{s}{s+t}h(\alpha)+\frac{t}{s+t}h(\beta)\leq h\left(\frac{s}{s+t}\alpha+\frac{t}{s+t}\beta\right).\]
Since $s$ and $t$ are two arbitrary positive integers, we have
\[\lambda h(\alpha)+(1-\lambda)h(\beta)\leq h(\lambda\alpha+(1-\lambda)\beta)\]	
for any $0<\lambda<1$. It means that the function $h(\alpha)$ is rational concavity.	

By the definition of $h(\alpha)$, for any $\eta>0$, there exists $\delta_0>0$ such that
\begin{align}\label{inequality h alpha}
\varlimsup_{n\to\infty}\frac{\log h(\alpha,n,\delta)}{(\log p)n}<h(\alpha)+\frac{\eta}{2}	
\end{align} 
for any $0<\delta<\delta_0$.  
Take a number $\gamma$ which satisfies $|\alpha-\gamma|<\delta/2$. Then, we have $H(\gamma,n,\delta/2)\subset H(\alpha,n,\delta)$. It yields that $h(\gamma,n,\delta/2)\leq h(\alpha,n,\delta)$. Moreover, by the definition of $h(\gamma)$, there exists some $\delta_1$ satisfying $0<\delta_1<\delta_0$ such that 
\[h(\gamma)\leq \varlimsup_{n\to\infty}\frac{\log h(\gamma,n,\delta/2)}{(\log p)n}+\frac{\eta}{2}\leq\varlimsup_{n\to\infty}\frac{\log h(\alpha,n,\delta)}{(\log p)n}+\frac{\eta}{2}.\]
for any $0<\delta<\delta_1$. This, together with \eqref{inequality h alpha}, yields that
$h(\gamma)<h(\alpha)+\eta$ if $|\alpha-\gamma|<\delta/2$, where  $0<\delta<\delta_1$. 
It implies that the function $h(\alpha)$ is upper semi-continuous. So, $h(\alpha)$ is concave and then continuous on $(0,p-1)$. 

Next, we show that $h(\alpha)$ is continuous at $\alpha=0$ and $\alpha=p-1$. Since
$h(\alpha,n,\delta)\leq\bar{h}(\alpha,n,\delta)$, similar to the estimation \eqref{inequality h zero} we have
\begin{align*} 
  \lim_{\alpha\to0^+}h(\alpha)&\leq\lim_{\alpha\to0^+}\lim_{\delta\to0}\varliminf_{n\to\infty}\frac{\log\big(n^{[n(\alpha+\delta)]}/[n(\alpha+\delta)]!\big)}{(\log p)n}\\
  &\leq\lim_{\alpha\to0^+}\lim_{\delta\to0}\frac{(\alpha+\delta)\log\frac{e}{(\alpha+\delta)}}{\log p}=\lim_{\alpha\to0^+}\frac{\alpha\log\frac{e}{\alpha}}{\log p}=0.  
\end{align*} 
It follows that $\lim_{\alpha\to0^+}h(\alpha)=0=h(0)$. Thus, $h(\alpha)$ is continuous at $\alpha=0$. Similarly, the continuity of $h(\alpha)$ at $\alpha=p-1$ holds as well. Thus, $h(\alpha)$ is concave and continuous on $[0,p-1]$. 

(3) The proof of the property $h(\alpha)=h(p-1-\alpha)$ is similar to the discussion for Theorem~\ref{theorem barh bh}. The monotonicity of $h(\alpha)$ is followed by the concavity of $h(\alpha)$ in property \eqref{halpha 2}.

(4) The former part is followed by the increasing property of $h(\alpha)$ in \eqref{halpha 3} and the inequality \[h(\alpha,n,\delta)\leq\bar{h}(\alpha,n,\delta)\leq2\left(\left[\frac{\alpha+\delta}{2\delta}\right]+1\right)h(\alpha,n,\delta),\] 
where $\alpha\leq(p-1)/2$. The second part can be dealt with in a similar way. 		
\end{proof}

\section{Besicovitch sets}
In this section, we will determine the Hausdorff dimensions of Bsicovitch sets $\b{\it E}(\alpha)$ and $\bar{E}(\alpha)$ and present at last a further property of the $p$-adic entropy function for ready use.

Let $x\in I$. Denote by 
\[\b{\it A}(x)=\varliminf_{n\to\infty}\frac{S_n(x)}{n}\quad\text{and}\quad \bar{A}(x)=\varlimsup_{n\to\infty}\frac{S_n(x)}{n}\] 
the lower and upper digit means of $x$, respectively. Let $0\leq\alpha\leq p-1$. Define the level sets
\begin{align}\label{E alpha}
   \b{\it E}(\alpha)=\left\{x\in I\colon \b{\it A}(x)\geq\alpha\right\}\quad\text{and}\quad\bar{E}(\alpha)=\left\{x\in I\colon \bar{A}(x)\leq\alpha\right\},
\end{align}
which are called the \emph{Besicovitch sets} in this paper. Note that the Hausdorff dimensions of $\b{\it E}(\alpha)$ and $\bar{E}(\alpha)$ are determined by Besicovitch~\cite{Be} in the case of binary expansion. For the present general case, we have

\begin{thm}\label{theorem general}
	Let $0\leq\alpha\leq p-1$. Then
	\begin{equation}\label{equation E alpha 1} 
	\dim_H\b{\it E}(\alpha)=\begin{cases} 1,\indent &0\leq\alpha<(p-1)/2;\\
	h(\alpha),\indent &(p-1)/2\leq\alpha\leq p-1,\end{cases}
	\end{equation}	
	and
	\begin{equation}\label{equation E alpha 2} 
    \dim_H\bar{E}(\alpha)=\begin{cases} h(\alpha),\indent &0\leq\alpha<(p-1)/2;\\
    1,\indent &(p-1)/2\leq\alpha\leq p-1.\end{cases}
    \end{equation}
\end{thm}

To prove this theorem, we would like to introduce a lemma about the Hausdorff dimensions of homogeneous Moran sets. Here, It is assumed the readers are familiar with the definition and structure of homogeneous Moran sets, for which one can see \cite{FWW} for more details. 

Let $\{N_k\}_{k\geq1}$ be a sequence of integers and $\{c_k\}_{k\geq1}$ be a sequence of positive numbers satisfying $N_k\geq2$, $0<c_k<1$ and $N_kc_k\leq1$. Let $\mathcal{M}=\mathcal{M}\big(I,\{N_k\}_{k\geq1},\{c_k\}_{k\geq1}\big)$ be the homogeneous Moran set determined by the sequences $\{N_k\}_{k\geq1}$ and $\{c_k\}_{k\geq1}$. Denote
\[s=\varliminf_{k\to\infty}\frac{\log(N_1N_2\cdots N_k)}{-\log(c_1c_2\cdots c_{k+1}N_{k+1})}.\]
Then we have
\begin{lem}[See Theorem 2.1 and Corollary 2.1 in \cite{FWW}]\label{lemma FWW}
	Let $\mathcal{M}$ be a homogeneous Moran set, then $\dim_H\mathcal{M}\geq s$. Moreover, if $\inf_{k\geq1}c_k>0$, then $\dim_H\mathcal{M}=s$.
\end{lem}

\begin{proof}[Proof of Theorem \ref{theorem general}]
We will give only the proof of \eqref{equation E alpha 2} since the conclusion \eqref{equation E alpha 1} can be dealt with in a similar way. 

For the second part of \eqref{equation E alpha 2}, it is obvious that 
$\{x\in I\colon\lim_{n\to\infty}S_n(x)/n=(p-1)/2\}\subset\bar{E}(\alpha)$. Since $\lim_{n\to\infty}S_n(x)/n=(p-1)/2$ for almost all $x\in I$ by the ergodic theorem, we have
\[\dim_H\bar{E}(\alpha)\geq\dim_H\left\{x\in I\colon\lim_{n\to\infty}\frac{S_n(x)}{n}=\frac{p-1}{2}\right\}=1.\]
It follows that $\dim_H\bar{E}(\alpha)=1$ as $(p-1)/2\leq\alpha\leq p-1$.

For the first part of \eqref{equation E alpha 2}, first we will show the upper bound of Hausdorff dimension of $\bar{E}(\alpha)$ is $h(\alpha)$ as $0\leq\alpha<(p-1)/2$. For any $\delta>0$, we have
\[\bar{E}(\alpha)\subset\bigcap_{l=1}^\infty\bigcup_{n=l}^\infty\bigcup_{x_1\cdots x_n\in\bar{H}(\alpha,n,\delta)}I(x_1\cdots x_n),\]
where the cylinder $I(x_1\cdots x_n)=\{y=0.y_1y_2\ldots\in I\colon y_1=x_1,\ldots,y_n=x_n\}$. By the definition of $\bar{h}(\alpha)$, for any $\eta>0$, there exists an integer $N$ such that 
\[\bar{h}(\alpha,n,\delta)<p^{n\left(\bar{h}(\alpha)+\frac{\eta}{2}\right)},\quad\forall n>N.\]	
Then, for any $l>N$, the $(\bar{h}(\alpha)+\eta)$-Hausdorff measure of $\bar{E}(\alpha)$ satisfies
\[\mathbb{H}_{p^{-l}}^{\bar{h}(\alpha)+\eta}\big(\bar{E}(\alpha)\big)\leq\sum_{n=l}^{\infty}\bar{h}(\alpha,n,\delta)(p^{-n})^s<\sum_{n=l}^{\infty}(p^{-\frac{\eta}{2}})^n<\infty.\]	
This implies that $\dim_H\bar{E}(\alpha)\leq\bar{h}(\alpha)+\eta$. Thus, $\dim_H\bar{E}(\alpha)\leq\bar{h}(\alpha)=h(\alpha)$ by the arbitrariness of $\eta$ and \eqref{halpha 4} of Theorem \ref{theorem enumerate 1}.

Next, we turn to show that the lower bound of the Hausdorff dimension of $\bar{E}(\alpha)$ is $h(\alpha)$. For this, we will prove $\dim_H\bar{E}(\alpha)\geq\tau$ for any $0<\tau<h(\alpha)$.

Since $\tau<h(\alpha)$, we can take two sequences, one is an increasing integer sequence $\{n_j\}_{j\geq1}$ and the other is a decreasing positive sequence $\{\delta_j\}_{j\geq1}$ satisfying $\lim_{j\to\infty}\delta_j=0$, such that
\[h(\alpha,n_j,\delta_j)>p^{n_j\tau}.\]
Let $j\geq1$ and write
\[F_j(\alpha)=\left\{x_1x_2\cdots x_{n_j}\in A^{n_j}\colon\left|\frac{\sum_{i=1}^{n_j}x_i}{n_j}-\alpha\right|<\delta_j\right\}.\]
Take a positive integer sequence $\{m_i\}_{i\geq1}$ satisfies
\[\lim_{j\to\infty}\frac{n_{j+1}}{\sum_{i=1}^{j}m_in_i}=0.\] Denote by \[q_j=m_1n_1+m_2n_2+\ldots+m_jn_j,\quad j\geq1.\] 
Based on the sequence of sets $\{F_j(\alpha)\}_{j\geq1}$, construct the Moran set
\begin{align*}
	\begin{split}
		\mathcal{F}(\alpha)&=\left\{0.x_1x_2\ldots\in I\colon x_{q_i+1}\cdots x_{q_{i+1}}\in F_i(\alpha)^{m_i},\forall i\geq1\right\}\\
		&=:0.\prod_{i=1}^{\infty}F_i(\alpha)^{m_i}.
	\end{split}
\end{align*}	
Here and in the sequel, if $F$ is a set of words with equal length and $m$ is a positive integer, then we use the notation $F^m$ to denote the set in which every word is the concatenations of $m$ words in the set $F$, and $F^\infty$ the set in which every sequence is the concatenations of infinite words in the set $F$. For the sequence of sets of words $\{F_i\}_{i\geq1}$,   $\prod_{i=1}^\infty F_i$ denotes the set in which every sequence is the successive concatenations of words in the set $F_i$ according to the order of natural numbers. 

It is easy to see that
\[\frac{m_1n_1(\alpha-\delta_1)+\cdots+m_jn_j(\alpha-\delta_j)}{m_1n_1+\cdots+m_jn_j}\leq\frac{S_{q_j}(x)}{q_j}\leq\frac{m_1n_1(\alpha+\delta_1)+\cdots+m_jn_j(\alpha+\delta_j)}{m_1n_1+\cdots+m_jn_j}.\]
Since $m_jn_j\to\infty$ and $\delta_j\to0$ as $j\to\infty$, we have
\[\lim_{j\to\infty}\frac{S_{q_j}(x)}{q_j}=\alpha\]
by the squeeze theorem. This implies the upper limit of $S_n(x)/n$ is $\alpha$. So, we have $\mathcal{F}(\alpha)\subset\bar{E}(\alpha)$ and then $\dim_H\bar{E}(\alpha)\geq\dim_H\mathcal{F}(\alpha)$. 

For any integer $n$ large enough, there exist two integers $j\geq1$ and $b$ such that 
\[0\leq b<m_{k+1}\quad\text{and}\quad\sum_{i=1}^jm_in_i+bn_{j+1}\leq n<\sum_{i=1}^jm_in_i+(b+1)n_{j+1}.\]
Then, by the first assertion of Lemma \ref{lemma FWW} we have
\[\dim_H\mathcal{F}(\alpha)\geq\varliminf_{j\to\infty}\frac{\big(\sum_{i=1}^jm_in_i+bn_{j+1}\big)\tau\log p}{\big(\sum_{i=1}^jm_in_i+(b+1)n_{j+1}\big)\log p-n_{j+1}\tau\log p}=\tau.\]
Thus, we obtain that $\dim_H\bar{E}(\alpha)\geq\tau$, which shows the first part of \eqref{equation E alpha 2}.

The proof is completed now.	
\end{proof}

Moreover, denote by 
\[A(x)=\lim_{n\to\infty}\frac{S_n(x)}{n},\quad x\in I,\] 
\emph{the arithmetic digit mean} of $x$ if the limit exists and define the level set related to it as
\begin{align}\label{E = alpha}
	E(\alpha)=\left\{x\in I\colon A(x)=\alpha\right\}.
\end{align}
Then, by the same technique used in the proof of Theorem~\ref{theorem general}, we may get

\begin{thm}\label{theorem = general}
	For any $0\leq\alpha\leq p-1$, we have that $\dim_HE(\alpha)=h(\alpha)$.
\end{thm}

\begin{cor}\label{corollary p-1 2}
  $h\big((p-1)/2\big)=1$.
\end{cor}
\begin{proof}
Since $A(x)=(p-1)/2$ for almost all $x\in I$, by Theorem~\ref{theorem = general} we have
\begin{align*}
	\begin{split}
		1=\dim_H\left\{x\in I\colon A(x)=\frac{p-1}{2}\right\} 
		=\dim_HE\left(\frac{p-1}{2}\right)=h\left(\frac{p-1}{2}\right). 
	\end{split}	
\end{align*}
It ends the proof.	
\end{proof}	

\begin{rem}\label{remark b alpha beta}
	By Corollary \ref{corollary p-1 2}, Theorem \ref{theorem main theorem} can be restated in details as follows: if $0\leq\alpha\leq\beta<(p-1)/2$, then $\dim_HB(\alpha,\beta)=h(\beta)$; if $0\leq\alpha\leq(p-1)/2\leq\beta\leq1$, then $\dim_HB(\alpha,\beta)=1$; if $(p-1)/2<\alpha\leq\beta\leq p-1$, then $\dim_HB(\alpha,\beta)=h(\alpha)$.
\end{rem}
	
\section{Some Moran sets}
In this section, we will introduce some Moran sets constructed by sets of words with bounded digit sums and then determine their Hausdorff dimensions. Based on them, we will construct suitable subsets to achieve the lower bound of Hausdorff dimension of $B(\alpha,\beta)$ in the last section.

Let $M\geq1$ be an integer. Take two integers $P$ and $Q$ satisfying $0\leq P\leq Q\leq (p-1)M$. Write
\[W([P,Q],M):=\left\{x_1x_2\cdots x_M\in A^M\colon P\leq\sum_{i=1}^Mx_i\leq Q\right\}.\]
Then define the Moran set 
\begin{align*}
    \mathcal {W}([P,Q],M)&:={}0.W([P,Q],M)^\infty.
\end{align*}
For the size of the set $\mathcal {W}([P,Q],M)$, by the second assertion in Lemma \ref{lemma FWW}, we can get immediately that 
\begin{lem}\label{lemma pqm} 
	Let $0\leq P\leq Q\leq (p-1)M$ and $M\geq1$. Then
	\[\dim_H\mathcal {W}([P,Q],M)=\frac{\log\card W([P,Q],M)}{(\log p)M}.\]
\end{lem}
Here and in the sequel, if $P=Q$, then write $W([P,Q],M)$ as $W(P,M)$ and $\mathcal {W}([P,Q],M)$ as $\mathcal {W}(P,M)$ respectively for simplicity.

Let $\alpha$ be a real number and $0\leq\alpha\leq p-1$. Let $n\geq1$. Define the function
\[\omega(\alpha,n)=\card W([\alpha n],n).\]
Then the corresponding properties in the following lemma is evident.

\begin{lem}\label{theorem enumerate 2}
   Let $0\leq\alpha\leq p-1$ and $n\geq1$. Then 
\begin{enumerate}
	\item\label{omega 1} For each $\alpha$, $\omega(\alpha,n)$ is increasing with respect to $n$;
	\item\label{omega 2} For each $n$, $\omega(\alpha,n)$ is constant on      $[(k-1)/n,k/n)$, where $1\leq k\leq(p-1)n$, with respect to $\alpha$;
    \item\label{omega 3} For each $n$, $\omega(\alpha,n)$ is increasing on $[0,(p-1)/2]+1/n)$ and decreasing on $[(p-1)/2+1/n,p-1]$ with respect to $\alpha$.
\end{enumerate}	   
\end{lem}

Moreover, we have

\begin{lem}\label{lemma logcard}
	Let $0\leq\alpha\leq p-1$. Then
	\[\lim_{n\to\infty}\frac{\log\card W([\alpha n],n)}{(\log p)n}=h(\alpha).\]
\end{lem}
\begin{proof}
	We first show that
\begin{align}\label{equality lemma logcard}
	\varliminf_{n\to\infty}\frac{\log\card W([\alpha n],n)}{(\log p)n}=h(\alpha).
\end{align}
	The proof is divided into three cases: $0\leq\alpha<(p-1)/2$, $\alpha=(p-1)/2$ and $(p-1)/2<\alpha\leq p-1$. Here, we give only the proof of the first case.
	
	Take $\delta>0$ such that $\alpha+\delta<(p-1)/2$. Since 
	\[\card W([(\alpha+\delta)n],n)\leq\frac{n^{n\delta}}{(n\delta-1)!}\card W([\alpha n],n),\]
	by the Stirling's approximation we have 
\begin{align*}
	\begin{split}
		\varliminf_{n\to\infty}\frac{\log\card W([(\alpha+\delta)n],n)}{(\log p)n}&\leq\varliminf_{n\to\infty}\frac{\log\frac{n^{n\delta}}{(n\delta-1)!}}{(\log p)n}+\varliminf_{n\to\infty}\frac{\log\card W([\alpha n],n)}{(\log p)n}\\
		&=\frac{\delta\log\frac{e}{\delta}}{\log p}+\varliminf_{n\to\infty}\frac{\log\card W([\alpha n],n)}{(\log p)n}. 
	\end{split}
\end{align*}	
Let $\delta\to0$ in both sides, since $\lim_{\delta\to0}(\delta\log\frac{e}{\delta})/\log p=0$, we have
\begin{align}\label{inequality log halpha}
		\lim_{\delta\to0}\varliminf_{n\to\infty}\frac{\log\card W([(\alpha+\delta)n],n)}{(\log p)n}\leq\varliminf_{n\to\infty}\frac{\log\card W([\alpha n],n)}{(\log p)n}.
\end{align}
Moreover, by the properties of function $\omega(\alpha,n)$ in Lemma~\ref{theorem enumerate 2}, we have
\begin{align*}
	\begin{split}
		\card W([(\alpha+\delta)n],n)&\leq\card W\big(\big[[(\alpha-\delta)n]+1,[(\alpha+\delta)n]\big],n\big)\\
		&=h(\alpha,n,\delta)\leq([2\delta n]+1)\card W([(\alpha+\delta)n],n).  
	\end{split}
\end{align*}
It follows that
\[\lim_{\delta\to0}\varliminf_{n\to\infty}\frac{\log\card W([(\alpha+\delta)n],n)}{(\log p)n}=\lim_{\delta\to0}\varliminf_{n\to\infty}\frac{\log h(\alpha,n,\delta)}{(\log p)n}=h(\alpha).\]	
This, together with \eqref{inequality log halpha}, yields that
\[h(\alpha)\leq\varliminf_{n\to\infty}\frac{\log\card W([\alpha n],n)}{(\log p)n}.\]
On the other hand, the inequality for the opposite direction is apparently true. So, the equality \eqref{equality lemma logcard} is established.

Since the equality $\varlimsup_{n\to\infty}\log\card W([\alpha n],n)/\big((\log p)n\big)=h(\alpha)$ can be proved as the above way, the proof of this lemma is finished now.
\end{proof}	

In the sequel, we will construct a Moran set $\mathcal{W}_M(\alpha)$, where $0\leq\alpha<p-1$, to obtain the lower bound of Hausdorff dimension of $B(\alpha,\beta)$ in Theorem~\ref{theorem main theorem}. At first, we construct recursively two sequences of sets of words $\{W_n(\alpha,M)\}_{n=1}^\infty$ and $\{V_n(\alpha,M)\}_{n=1}^\infty$ below. 

Let $M$ be sufficiently large such that $[\alpha M]+1<(p-1)M$. For brevity, write
\[W_1(\alpha,M)=W([\alpha M],M),\quad V_1(\alpha,M)=W([\alpha M]+1,M).\]
Suppose that the sets $W_i(\alpha,M)$ and $V_i(\alpha,M)$ are well-defined for all $1\leq i\leq n$,
then define 
\begin{align*}
	\begin{split}
		W_{n+1}(\alpha,M)=\big\{&x_1\cdots x_{2^nM}\in W([\alpha 2^nM],2^nM)\colon\\
		&x_{2^{n-1}Mi+1}\cdots x_{2^{n-1}M(i+1)}\in W_n(\alpha,M)\cup V_n(\alpha,M),i=0,1\big\},
	\end{split}
\end{align*}
\begin{align*}
	\begin{split}
		V_{n+1}(\alpha,M)=\big\{&x_1\cdots x_{2^nM}\in W([\alpha 2^nM]+1,2^nM) \colon \\&x_{2^{n-1}Mi+1}\cdots x_{2^{n-1}M(i+1)}\in W_n(\alpha,M)\cup V_n(\alpha,M),i=0,1\big\}.
	\end{split}
\end{align*}
The above definitions are valid since the estimation
\begin{equation}\label{cc5}
2[\alpha2^kM]<[\alpha2^{k+1}M]+1\leq 2\left([\alpha2^kM]+1\right)
\end{equation}
holds for any $0\leq\alpha<p-1$ and $k\geq0$.

With this construction, we know that for each $n\geq1$, every word in $W_n(\alpha,M)$ is of length $2^{n-1}M$ and the sum of elements is $[\alpha2^{n-1}M]$. Similarly, every word in $V_n(\alpha,M)$ is of length $2^{n-1}M$ and the sum of elements is $[\alpha2^{n-1}M]+1$. Moreover, we even have
%in this note we call $W_n(\alpha,M)$ and $V_n(\alpha,M)$, $n\geq0$, $\alpha$-blocks, and the words in them $\alpha$-words if there is no confusion.

\begin{rem}\label{remark decompose}
	For any $0\leq i\leq n-1$, we can decompose uniquely each word in $W_n(\alpha,M)$ and $V_n(\alpha,M)$ into successive concatenations of $2^iM$-words, the sum of elements in each $2^iM$-word is $[\alpha 2^iM]$ or $[\alpha 2^iM]+1$.
\end{rem} 

Based on the family of sets of $\alpha$-words $\{W_n(\alpha,M)\}_{n=1}^\infty$, define the Moran set
\begin{align*}
		\mathcal{W}_M(\alpha):=0.\prod_{n=1}^\infty W_n(\alpha,M).
\end{align*}
Then we have

\begin{lem}\label{lemma alpha M}
	Let $0\leq\alpha<p-1$, then
	\begin{align}\label{equantion alpha M}
		\lim_{M\to\infty}\dim_H\mathcal{W}_M(\alpha)=h(\alpha).
	\end{align}
\end{lem}
\begin{proof}
For the case $0\leq\alpha<(p-1)/2$, take $M$ to be large enough such that $[\alpha M]+1<[(p-1)M/2]$. By the monotonicity of function $\omega$ in \eqref{omega 3} of Lemma~\ref{theorem enumerate 2}, we have 
\[\card W([\alpha M],M)=\omega(\alpha,M)\leq \omega(\alpha+1/M,M)=\card W([\alpha M]+1,M).\]
From this and the structures of words in $\mathcal{W}_M(\alpha)$ in Remark~\ref{remark decompose}, we know that
\begin{align}\label{inequality alpha M}
	\dim_H\mathcal{W}([\alpha M],M)\leq\dim_H\mathcal{W}_M(\alpha)\leq\dim_H\mathcal{W}\big(\big[[\alpha M],[\alpha M]+1\big],M\big).
\end{align}
Moreover, by Lemma \ref{lemma pqm}, we have
\[\dim_H\mathcal{W}([\alpha M],M)=\frac{\log\card W([\alpha M],M)}{(\log p)M}\]
and
\begin{align*}
	\begin{split}
		\dim_H\mathcal{W}\big(\big[[\alpha M],[\alpha M]+1\big],M\big) 
		=\frac{\log\big(\card W([\alpha M],M)+\card W([\alpha M]+1,M)\big)}{(\log p)M}.
	\end{split}
\end{align*}
Thus, by Lemma~\ref{lemma logcard} we may obtain that
\[\lim_{M\to\infty}\dim_H\mathcal{W}([\alpha M],M)=\lim_{M\to\infty}\dim_H\mathcal{W}\big(\big[[\alpha M],[\alpha M]+1\big],M\big)=h(\alpha).\]
This, together with \eqref{inequality alpha M}, leads to the conclusion~\eqref{equantion alpha M}.

On the other hand, for the case $(p-1)/2\leq\alpha<p-1$, we can get similarly that
\[\dim_H\mathcal{W}([\alpha M]+1,M)\leq\dim_H\mathcal{W}_M(\alpha)\leq\dim_H\mathcal{W}\big(\big[[\alpha M],[\alpha M]+1\big],M\big)\]
for sufficiently large $M$. The proof of the remainder is similar to the first case. We omit the details here.

The proof is completed now.
\end{proof}

\section{Proof of Theorem \ref{theorem main theorem}}

In this section, we will give the proof of Theorem~\ref{theorem main theorem}. 
First, we would like to present the following lemma to reveal the relations among the lower and upper digit means of $x\in I$ and the lower and upper moving digit means of $x\in I$, which will be used to achieve the upper bound of Hausdorff dimension of $B(\alpha,\beta)$. 

\begin{lem}\label{lemma relation}
	For any $x\in I$, we have
	\begin{equation} 
	\b{\it M}(x)\leq\b{\it A}(x)\leq\bar{A}(x)\leq\bar{M}(x).
	\end{equation}	
\end{lem}
\begin{proof}
	It suffices to show that $\bar{A}(x)\leq\bar{M}(x)$. Write $\bar{M}(x)=\beta\in[0,p-1]$. Then, for any $\epsilon>0$, there exists an integer $N>0$, such that
	\[\varlimsup\limits_{m\to\infty}\frac{S_n(T^mx)}{n}<\beta+\epsilon,\quad\forall n\geq N.\]
	Furthermore, there exists an integer $Q=Q(n)\geq1$ such that
	\[\frac{S_n(T^mx)}{n}<\beta+\epsilon,\quad \mbox{i.e.},\quad \sum_{i=1}^nx_{m+i}\leq[n(\beta+\epsilon)]\]
	for any $ m\geq Q$. Suppose $t\geq Q$ and $t=Q+r+kn$, where $0\leq r\leq n-1$, then
	\[\frac{S_t(x)}{t}\leq\frac{Q+r+k[n(\beta+\epsilon)]}{Q+r+kn}.\]
	Let $t\to\infty$, then $k\to\infty$ and
	\[\varlimsup_{t\to\infty}\frac{S_t(x)}{t}\leq\frac{[n(\beta+\epsilon)]}{n}\leq\beta+\epsilon.\]
	Since $\epsilon$ is arbitrary, we have $\bar{A}(x)\leq\beta=\bar{M}(x)$.  
\end{proof}

Next, we would like to present a lemma for the computation of lower bound of the Hausdorff dimension of $B(\alpha,\beta)$. Let $\mathbb{M}$ be a subset of $\mathbb{N}$. We say the set $\mathbb{M}$ is of density $\rho\in[0,1]$ in $\mathbb{N}$ if
\[\lim_{n\to\infty}\frac{\card\{i\in\mathbb{M}\colon i\leq	n\}}{n}=\rho.\]  
Write
$\mathbb{N}\backslash\mathbb{M}=\{n_i\}_{i\geq1}$ where $n_i<n_{i+1}$ for all $i\geq1$. Define a
mapping $\varphi_\mathbb{M}\colon I\to I$ by
\[0.x_1x_2\ldots\mapsto 0.x_{n_1}x_{n_2}\ldots.\] 
Under the mapping $\varphi_\mathbb{M}$, for any given subset $D\subset I$, we may obtain another set $\varphi_\mathbb{M}(D)=\{\varphi_\mathbb{M}(x)\colon x\in D\}$. 
Moreover, we have

\begin{lem}[See Lemma 2.3 in~\cite{CT}]\label{lemma invariance}
	Suppose that the set $\mathbb{M}$ is of density zero in $\mathbb{N}$. Then for any set $D\subset I$ we have $\dim_HD=\dim_H\phi_\mathbb{M}(D)$.
\end{lem}

Lemma~\ref{lemma invariance} implies that for a set $D$, its Hausdorff dimension is invariant after deleting the digits, for which the set of their positions is of density zero in $\mathbb{N}$, from the sequences of numbers in $D$.

Now, we are ready to give the proof of Theorem~\ref{theorem main theorem}.

\begin{proof}[Proof of Theorem~\ref{theorem main theorem}]
According to Remark \ref{remark b alpha beta}, the proof is divided into three parts:
\begin{enumerate}
	    \item $\dim_HB(\alpha,\beta)=h(\beta)$ if $0\leq\alpha\leq\beta<(p-1)/2$;
		\item $\dim_HB(\alpha,\beta)=1$ if $\alpha\leq(p-1)/2\leq\beta$;
		\item $\dim_HB(\alpha,\beta)=h(\alpha)$ if $(p-1)/2<\alpha\leq\beta\leq p-1$.
\end{enumerate}
In the following, we will give the proofs of them respectively.	
	
(1) For the upper bound, by Lemma~\ref{lemma relation}, we have $B(\alpha,\beta)\subset\bar{S}(\beta)$. It follows that $\dim_HB(\alpha,\beta)\leq\dim_H\bar{S}(\beta)=h(\beta)$. 	
	
For the lower bound, construct the set	
\[\mathcal{W}_M(\alpha,\beta)=0.\prod_{n=1}^{\infty}\big(W_n(\alpha,M)\times W_n(\beta,M)^n\big).\]		
Then we have: (a) $\mathcal{W}_M(\alpha,\beta)\subset B(\alpha,\beta)$; (b) $\dim_H\mathcal{W}_M(\alpha,\beta)=\dim_H\mathcal{W}_M(\beta)$. 

For the proof of (a), note that for any $i\geq1$, each word in the set
\[\prod_{n=1}^i\big(W_n(\alpha,M)\times W_n(\beta,M)^n\big),\] 
is of length $i2^iM$ and the length of words in $W_n(\alpha,M)$ and $W_n(\beta,M)$, $n>i$, are of common lengths $2^{n-1}M$. So,
we may decompose every number $x\in\mathcal{W}_M(\alpha,\beta)$ into successive concatenations of $2^iM$-words. Take $n$ to be sufficiently large and write $n=k2^iM+r$, $0\leq r\leq 2^iM-1$. 
Then  
\[\frac{(k-1)[\alpha2^iM]}{k2^iM+r}\leq\b{\it M}_n(x)\leq\frac{(k+2)([\alpha2^iM]+1)}{k2^iM+r}.\] 
Let $n\to\infty$, then $k\to\infty$. It yields that 
\[\frac{[\alpha2^iM]}{2^iM}\leq\b{\it M}(x)\leq\frac{[\alpha2^iM]+1}{2^iM}.\] 
Since this inequality holds for all $i\geq1$, we have $\b{\it M}(x)=\alpha$ by letting $i\to\infty$. The other conclusion $\bar{M}(x)=\beta$ can be deduced in a similar manner. So, we have $\mathcal{W}_M(\alpha,\beta)\subset B(\alpha,\beta)$.

For the second assertion (b), the set of positions occupied by the words in sets $W_n(\alpha,M)$, $n\geq1$, is of density zero for any $x\in\mathcal{W}_M(\alpha,\beta)$. 
By deleting all these words in the sequences of numbers in $\mathcal{W}_M(\alpha,\beta)$, we obtain the set $\mathcal{W}_M(\beta)$.
Then, (b) is established by Lemma~\ref{lemma invariance}. 

By (a) and (b), we obtain that
\[\dim_HB(\alpha,\beta)\geq\dim_H\mathcal{W}_M(\alpha,\beta)=\dim_H\mathcal{W}_M(\beta).\]
Let $M\to\infty$, it yields that $\dim_HB(\alpha,\beta)\geq h(\beta)$ according to Lemma~\ref{lemma alpha M}. 
The proof of this part is finished.

(2) This part is split into four cases for consideration: i) $\alpha< (p-1)/2<\beta$; ii) $\alpha=(p-1)/2<\beta$; iii) $\alpha<\beta=(p-1)/2$; iv) $\alpha=\beta=(p-1)/2$.

Case i): $\alpha< (p-1)/2<\beta$. Since $h((p-1)/2)=1$, for $\epsilon$ small enough, there exists $\delta_0>0$ and $n_0>0$ such that 
\[\alpha<\frac{p-1}{2}-\delta_0<\frac{p-1}{2}+\delta_0<\beta\quad\text{and}\quad\frac{\log h\left(\frac{p-1}{2},n_0,\delta_0\right)}{(\log p)n_0}>1-\epsilon.\]
Based on the set
\[H\left(\frac{p-1}{2},n_0,\delta_0\right)=\bigg\{x_1\cdots x_{n_0}\in A^{n_0}\colon n_0(\frac{p-1}{2}-\delta_0)<\sum_{i=1}^{n_0}x_i<n_0(\frac{p-1}{2}+\delta_0)\bigg\},\]
define
\begin{align}\label{definition u n0}
	\mathcal{H}\left(\frac{p-1}{2},n_0,\delta_0\right)=0.H\left(\frac{p-1}{2},n_0,\delta_0\right)^\infty.
\end{align}
Then, by Lemma \ref{lemma pqm}, we have 
\[\dim_H\mathcal{H}\left(\frac{p-1}{2},n_0,\delta_0\right)=\frac{\log\card H\left(\frac{p-1}{2},n_0,\delta_0\right)}{(\log p)n_0}=\frac{\log h\left(\frac{p-1}{2},n_0,\delta_0\right)}{(\log p)n_0}>1-\epsilon.\]
Now, construct the set
\begin{align}\label{definition w n0}
	\mathcal{W}_{M,n_0,\delta_0}(\alpha,\beta)=0.\prod_{n=1}^{\infty}\left(W_n(\alpha,M)\times H\left(\frac{p-1}{2},n_0,\delta_0\right)^{3^{n-1}}\times W_n(\beta,M)\right).
\end{align}
Similar to the proof of the foregoing part, we can also deduce that 
\[\mathcal{W}_{M,n_0,\delta_0}(\alpha,\beta)\subset B(\alpha,\beta)\quad\text{and}\quad\dim_H\mathcal{W}_{M,n_0,\delta_0}(\alpha,\beta)=\dim_H\mathcal{H}\left(\frac{p-1}{2},n_0,\delta_0\right).\]
Thus, we have \[\dim_HB(\alpha,\beta)\geq\dim_H\mathcal{H}\left(\frac{p-1}{2},n_0,\delta_0\right)>1-\epsilon.\] 
It proves this case since $\epsilon$ is arbitrary.

Case ii): $\alpha=(p-1)/2<\beta$. In this case, take $\delta_0$ satisfying $(p-1)/2<(p-1)/2+\delta_0<\beta$. Put
\[H'\left(\frac{p-1}{2},n_0,\delta_0\right)=\bigg\{x_1\cdots x_{n_0}\in A^{n_0}\colon \frac{p-1}{2}n_0<\sum_{i=1}^{n_0}x_i<(\frac{p-1}{2}+\delta_0)n_0\bigg\}\]
and
\[\mathcal{H}'\left(\frac{p-1}{2},n_0,\delta_0\right)=0.H'\left(\frac{p-1}{2},n_0,\delta_0\right)^\infty.\]
Then, similar to the proof of Lemma \ref{lemma logcard}, we have 
\[\lim_{\delta_0\to0}\varliminf_{n_0\to\infty}\frac{\log\card H'\left(\frac{p-1}{2},n_0,\delta_0\right)}{(\log p)n_0}=\lim_{\delta_0\to0}\varliminf_{n_0\to\infty}\frac{\log\card H\left(\frac{p-1}{2},n_0,\delta_0\right)}{(\log p)n_0}=1.\]
So, for any $\epsilon>0$, there exists $n_0$ and $\delta_0$ such that \[\dim_H\mathcal{H}'\left(\frac{p-1}{2},n_0,\delta_0\right)>1-\epsilon.\] 
Next, construct the Moran set $\mathcal{W}_{M,n_0,\delta_0}'\big((p-1)/2,\beta\big)$ as $\mathcal{W}_{M,n_0,\delta_0}(\alpha,\beta)$ in \eqref{definition w n0} by replacing $H\big((p-1)/2,n_0,\delta_0\big)$ with $H'\left((p-1)/2,n_0,\delta_0\right)$. 
For the remaining proof of this case, it is just similar to the above discussion in Case i).  

Case iii): $\alpha<\beta=(p-1)/2$. It can be proved as that of Case ii).

Case iv): $\alpha=\beta=(p-1)/2$. In this case, take $n_0$ to be even and 
consider the set
\[\mathcal{W}_{n_0}\left(\frac{p-1}{2}\right)=0.W\left(\frac{p-1}{2},n_0\right)^\infty,\]
where
\[W\left(\frac{p-1}{2},n_0\right)=\bigg\{x_1\cdots x_{n_0}\in A^{n_0}\colon \sum_{i=1}^{n_0}x_i=\frac{p-1}{2}n_0\bigg\}.\]
Then we have
\[\lim_{n_0\to\infty}\dim_H\mathcal{W}_{n_0}\left(\frac{p-1}{2}\right)=h\left(\frac{p-1}{2}\right)=1\]
by Lemma \ref{lemma logcard}. It is evident that \[\mathcal{W}_{n_0}\left(\frac{p-1}{2}\right)\subset B\left(\frac{p-1}{2},\frac{p-1}{2}\right)=B\left(\frac{p-1}{2}\right).\]
Thus, $\dim_HB\big((p-1)/2\big)\geq\dim_H\mathcal{W}_{n_0}\big((p-1)/2\big)$. Let $n_0\to\infty$, then we have $\dim_HB\big((p-1)/2\big)=1$.

(3) Three cases will be considered in this part. 

Case 1): $(p-1)/2<\alpha\leq\beta<p-1$. This case can be proved as that of part (1) and we omit the details here.

Case 2): $(p-1)/2<\alpha<\beta=p-1$. First, by Lemma~\ref{lemma relation}, we may obtain that $B(\alpha,p-1)\subset\b{\it E}(\alpha)$. So, $\dim_HB(\alpha,p-1)\leq\dim_H\b{\it E}(\alpha)=h(\alpha)$.

Second, construct the set
\[\mathcal{W}_M(\alpha,p-1)=0.\prod_{n=1}^{\infty}\big(W_n(\alpha,M)\times(p-1)^n\big),\]
where $(p-1)^n$ means the word $(p-1)\cdots(p-1)$ of length $n$. Then we can deduce similarly that 
\[\mathcal{W}_M(\alpha,p-1)\subset B(\alpha,p-1)\quad\text{and}\quad \dim_H\mathcal{W}_M(\alpha,p-1)=\dim_H\mathcal{W}_M(\alpha).\]
Thus, $\dim_HB(\alpha,p-1)\geq\dim_H\mathcal{W}_M(\alpha)=h(\alpha)$.

The above two assertions imply that $\dim_HB(\alpha,p-1)=h(\alpha)$.

Case 3): $\alpha=\beta=p-1$. Since $B(p-1,p-1)\subset\b{\it E}(p-1)$ and $\dim_H\b{\it E}(p-1)=h(p-1)=0$, we obtain that $\dim_HB(p-1,p-1)=0=h(p-1)$.

The proof is finished now.
\end{proof}	

At last, it should be pointed out that we can even study the \emph{moving digit mean} of $x$:
\begin{align}
	M(x)=\lim_{n\to\infty}\lim_{m\to\infty}\frac{S_n(T^mx)}{n},\quad x\in I.
\end{align}	
Moreover, define the level sets related to it as
\begin{align}
	B^\ast(\alpha)=\left\{x\in I\colon M(x)=\alpha\right\},\quad0\leq\alpha\leq p-1.
\end{align}
In this situation, the set $B^\ast(\alpha)$ is somewhat trivial because we have
\begin{thm}
	Let $0\leq\alpha\leq p-1$. If $\alpha=0,1,\ldots,p-1$, then $B^\ast(\alpha)$ is a countable set. Otherwise, $B^\ast(\alpha)$ is an empty set. Hence, we always have 
	\begin{align}
		\dim_HB^\ast(\alpha)=0
	\end{align} 
	for any $0\leq\alpha\leq p-1$.
\end{thm}
\begin{proof}
  Let $\alpha=i$, where $i=0,1,\ldots, p-1$. Then each number in $B^\ast(\alpha)$ is ultimately 1-periodic ending with $i^\infty$, that is
  \[B^\ast(\alpha)=\{x\in I\colon x=0.x_1x_2\ldots x_niii\ldots,n\geq 1\}.\]	
  Thus, $B^\ast(\alpha)$ is countable. On the other hand, if $i<\alpha<i+1$, where $i=0,1,\ldots$, or $p-2$, then for any $x\in B^\ast(\alpha)$ and $n\geq1$ the limit $\lim_{m\to\infty}S_n(T^mx)/n$ does not exist according to the proof by contradiction and Cauchy's criterion for convergence of sequences. It follows that $B^\ast(\alpha)=\emptyset$.  	
\end{proof}	

\subsection*{Acknowledgment}
This work was finished when the author visited the Laboratoire d'Analyse et de Math\'{e}matiques Appliqu\'{e}es, Universit\'{e} Paris-Est Cr\'{e}teil Val de Marne, France. Thanks a lot for the great encouragement from his collaborator and the assistance provided by the laboratory. 
%This work was supported by the scientific research project of Hubei Provincial Department of Education (No. B2017604) and the state scholarship fund of the China Scholarship Council.


\begin{thebibliography}{99}	
  \bibitem{Be} A.~Besicovitch, On the sum of digits of real numbers represented in the dyadic system, {\it Math. Ann.}, \textbf{110} (1934), 321--330.
  \bibitem{BSS} L. Barreira, B. Saussol and J. Schmeling, Distribution of frequencies of digits via multifractal analysis, {\it J. Number
  	Theory}, \textbf{97} (2002), 410--438.
  \bibitem{Bi} P. Billingsley, {\it Ergodic Theory and Information},
  Wiley, New York, 1965.
  \bibitem{CT} H. B. Chen and J. M. Tang, The waiting spectra of the sets      described by the quantitative waiting time indicators. {\it Sci. China Math.}, \textbf{11} (2014), no. 57, 2335--2346.  
  \bibitem{CWX} H. B. Chen, Z. X. Wen and Y. Xiong, Dimensions of level sets related to tangential dimensions. {\it J. Math. Anal. Appl.}, \textbf{340} (2008),  959--967.
  \bibitem{E} H.G.~Eggleston, The fractional dimension of a set defined by decimal properties, {\it Quart. J. Math. Oxford Ser.}, \textbf{20} (1949), 31--36. 
  \bibitem{F97} K. J.~Falconer, {\it Fractal Geometry-Mathematical Foundations and Applications}, John Wiley, 1990.
  \bibitem{FF} A. H. Fan and D. J. Feng, On the distribution of long-term time averages on symbolic space,  {\it J. Stat. Phys.}, \textbf{99} (2000), 813--856.
  \bibitem{FFW} A. H. Fan, D. J. Feng and J. Wu, Recurrence, dimension and	entropy, {\it J. London Math. Soc.}, \textbf{64} (2001), no. 1, 229--244.
  \bibitem{FWW} D. J. Feng, Z. Y. Wen and J. Wu, Some dimensional results for homogeneous Moran sets, {\it Sci. China Math.}, \textbf{40} (1997), 172--178.
  \bibitem{GI1} D. Guido and T. Isola. Tangential dimensions I. Metric spaces, {\it Houston Journal Math.}, \textbf{31} (2005), 1023--1045.  
  \bibitem{GI2} D. Guido and T. Isola, Tangential dimensions II.
  Measures, {\it Houston Journal Math.}, \textbf{32} (2006), 423--444.
  \bibitem{TWWX} B. Tan, B. W. Wang, J. Wu and J. Xu, Localized Birkhoff average in beta dynamical systems, {\it Discrete Contin. Dyn. Syst.}, \textbf{33} (2013), 2547--2564.
\end{thebibliography}
\end{document}